\documentclass[a4paper]{amsart}
\usepackage{amssymb,amsxtra}

\theoremstyle{plain}
\newtheorem{theorem}{Theorem}[section]
\newtheorem{lemma}[theorem]{Lemma}
\newtheorem{corollary}[theorem]{Corollary}
\newtheorem{proposition}[theorem]{Proposition}
\theoremstyle{remark}
\newtheorem{remark}[theorem]{Remark}

\newcommand{\GF}[1]{\mathbb{F}_{#1}}
\newcommand{\SL}[2]{\mathrm{SL}_{#1}(#2)}
\newcommand{\GL}{\mathrm{GL}}

\begin{document}
\title[Modular quotient singularities]{Three-dimensional isolated quotient singularities \\
        in even characteristic}
\author{Vladimir Shchigolev}
\author{Dmitry Stepanov}
\address[Shchigolev]{Financial University under the Government of the Russian Federation, 49 Leningradsky Prospekt, Moscow, Russia}
\email{shchigolev\_vladimir@yahoo.com}
\address[Stepanov]{The Department of Mathematical Modelling \\
         Bauman Moscow State Technical University \\
         2-ya Baumanskaya ul. 5, Moscow 105005, Russia}
\email{dstepanov@bmstu.ru}

\thanks{The first author was supported by RFBR grant no. 16-01-00756. The second author was supported by RFBR grants no. 14-01-00160 and no. 15-01-02164}
\date{}

\begin{abstract}
This paper is a complement to the work of the second author on modular quotient singularities in
odd characteristic. Here we prove that if $V$ is a three-dimensional vector space over a field of
characteristic $2$ and $G<\GL(V)$ is a finite subgroup generated by pseudoreflections and possessing
a $2$-dimensional invariant subspace $W$ such that the restriction of $G$ to $W$ is isomorphic to the
group $\SL{2}{\GF{2^n}}$, then the quotient $V/G$ is non-singular. This, together with
earlier known results on modular quotient singularities, implies first that a theorem of Kemper and Malle
on irreducible groups generated by pseudoreflections generalizes to reducible groups in dimension three,
and, second, that the classification of three-dimensional isolated singularities which are quotients of 
a vector space by a linear finite group reduces to Vincent's classification of non-modular isolated 
quotient singularities.
\end{abstract}

\maketitle

\section{Introduction}

Let $k$ be a field of characteristic $p$ and $V$ a finite dimensional vector space over $k$.
A linear map $\varphi\colon V\to V$ is called a \emph{pseudoreflection} if the set of points
fixed by $\varphi$ is a hyperplane in $V$. A pseudoreflection $\varphi$ is called a
\emph{transvection} if $1$ is the only eigenvalue of $\varphi$. Denote by $V^*$ the dual
space and by $S(V^*)$ its symmetric algebra. In \cite{KM} Kemper and Malle proved the
following theorem.
\begin{theorem}\label{T:KM}
Let $G$ be a finite irreducible subgroup of $\GL(V)$. Then its ring of invariants $S(V^*)^G$
is polynomial if and only if $G$ is generated by pseudoreflections and the pointwise stabilizer
in $G$ of any non-trivial subspace of $V$ has a polynomial ring of invariants.
\end{theorem}
Kemper and Malle also asked if the condition ``irreducible'' could be eliminated from the statement
of their theorem. They showed that to obtain such a generalization it is sufficient to investigate
the general reducible but non-decomposable case and pointed out that the generalized theorem
holds in dimension $2$. Note that the direct statement of Theorem~\ref{T:KM} (``if the ring
$S(V^*)$ is polynomial, then ...'') is correct without the condition of irreducibility; it follows
from the Chevalley-Shephard-Todd Theorem if $p$ does not divide the order of $G$, and
in the modular case $p\,|\,|G|$ it was proven by Serre.

From the perspective of singularity theory, Stepanov in \cite{Stepanov} showed that if the
generalized (to reducible groups $G$) theorem of Kemper and Malle is correct, it can be
interpreted as saying that each isolated singularity which is a quotient of a vector space by a
finite modular linear group is in fact isomorphic to a quotinet by a non-modular group. Thus the
classification of such singularities reduces to the known Vincet's classification of isolated quotient
singularities in the non-modular case; for the details, see \cite{Stepanov} and references
therein. Stepanov started also studying $3$-dimensional case and obtained the following result.
\begin{theorem}[{\cite[Theorem~4.1]{Stepanov}}]\label{T:odd}
Let $V$ be a $3$-dimensional vector space over an algebraically closed field
of characteristic $p$. Let $G$ be a finite subgroup of $GL(V)$ generated by pseudoreflections.
Denote by $G_p$ the normal subgroup of $G$ generated by all elements of order $p^r$, $r\geq 1$.
Assume that $G_p$ is either
\begin{enumerate}
   \item irreducible on $V$, or
   \item has a $1$-dimensional invariant subspace $U$, or
   \item has a $2$-dimensional invariant subspace $W$ and the restriction of $G_p$
   to $W$ is generated by two non-commuting transvections (and thus is irreducible).
\end{enumerate}
Then the generalized Kemper-Malle Theorem holds for $G$. Moreover, if $G$ satisfies
condition (3) or condition (2) plus the induced action of $G_p$ on $V/U$ is generated by two
non-commuting transvections, then $V/G$ is non-singular.
\end{theorem}

Note that if a map $\varphi\in\GL(W)$, $\dim W=2$, has order $p^r$, $r\geq 1$, then it has
order $p$ and is a transvection. In view of the classification of $2$-dimensional groups
generated by transvections, Theorem~\ref{T:odd} applies to all modular groups in odd
characteristic. In characteristic $2$ it remains to consider only the case when $G$ has a
$2$-dimensional invariant subspace $W$ and the restriction $H$ of $G_2$ to $W$ is isomorphic
to the group $\SL{2}{\GF{2^n}}$ (the group of all $2\times 2$ matrices of determinant $1$
with entries in the Galois field with $2^n$ elements), $n>1$, in its natural representation.

In the present paper we fill this gap and show, moreover, that no singularities arise
in the remaining case $H=\SL{2}{\GF{2^n}}$, $n>1$. Our main result is Theorem~\ref{T:even}
below. As was shown in \cite{Stepanov}, we can assume from the beginning that $G=G_2$ and
the base field $k$ is algebraically closed.

\begin{theorem}\label{T:even}
Let $V$ be a $3$-dimensional vector space over an algebraically closed field $k$ of characteristic
$2$. Let $G$ be a finite subgroup of $\GL(V)$ generated by pseudoreflections of order
$2^r$, $r\geq 1$, and hence by transvections. Assume that $G$ has a $2$-dimensional invariant
subspace $W$ and the restriction of $G$ to $W$ is isomorphic to the group $\SL{2}{\GF{2^n}}$,
$n>1$, in its natural representation. Then the ring of invariants $S(V^*)^G$ is polynomial.
\end{theorem}

\begin{remark}
It follows from our results that if $G<\GL(V)$, $\dim V=3$, characteristic is arbitrary, is
any finite subgroup generated by pseudoreflections and possessing a $2$-dimensional invariant
subspace or a $1$-dimensional invariant subspace satisfying the additional condition of
Theorem~\ref{T:odd}, then the quotient $V/G$ is non-singular. However, it is not true that
Chevalley-Shephard-Todd Theorem holds for modular groups in dimension $3$. In \cite{KM}
Kemper and Malle give examples of \emph{irreducible} groups $G$ generated by pseudoreflections
for which the ring $S(V^*)^G$ is not polynomial. In dimension $4$, there are examples (see
\cite[Example~11.0.3]{CW}) of reducible groups generated by pseudoreflections with singular quotients.
For general reducible $3$-dimensional groups $G$ generated by pseudoreflections, we do not
know if the quotient $V/G$ can be singular.
\end{remark}

As we explained above, our results and Theorem~\ref{T:KM} of Kemper and Malle imply the
following corollaries.
\begin{corollary}
The generalized Kemper-Malle Theorem holds in dimension $3$, i.e., if $V$ is a $3$-dimensional
vector space and $G<\GL(V)$ is \emph{any} finite subgroup, then the ring of invariants
$S(V^*)^G$ is polynomial if and only if $G$ is generated by pseudoreflections and the pointwise
stabilizer in $G$ of any non-trivial subspace of $V$ has a polynomial ring of invariants.
\end{corollary}

\begin{corollary}
If $V$ is a $3$-dimensional vector space over an arbitrary field $k$, and $G$ a finite subgroup
of $\GL(V)$ such that the variety $V/G$ has isolated singularity, then $V/G$ is isomorphic to one
of the non-modular isolated quotient singularities from Vincent's classification.
\end{corollary}

We prove our Theorem~\ref{T:even} by a more or less direct computation of the ring of
invariants of the group $G$. The proof is contained below in Sections~\ref{S:extofSL} and \ref{S:codim1}.

\section{Proof of Theorem~\ref{T:even}: the group $G$ as an extension of  $\SL{2}{\GF{2^n}}$}\label{S:extofSL}

Assume that a group $G$ satisfies the conditions of Theorem~\ref{T:even}, i.e.,
$G$ is generated by transvections, acts on a $3$-dimensional vector space $V$ with a $2$-dimensional
invariant subspace $W$, and the restriction of $G$ to $W$ is isomorphic to the natural action of
the group $\SL{2}{\GF{2^n}}$ on the space $k^2$ of column vectors. We shall fix a basis
$(e_1,e_2,e_3)$ of $V$ such that $e_1$ and $e_2$ span $W$ and each element of the group $G$
is represented in this basis by a matrix
$$\begin{pmatrix}
a & b & \alpha \\
c & d & \beta \\
0 & 0 & 1
\end{pmatrix},
$$
where $a,b,c,d\in\GF{2^n}\subset k$, $ad+bc=1$, $\alpha,\beta\in k$. We have an exact sequence
of groups
\begin{equation}\label{E:sequence}
0 \to N \to G \to \SL{2}{\GF{2^n}} \to 1,
\end{equation}
where $N$ is the kernel of the natural restriction map. In our basis, $N$ consists of the matrices
$$\begin{pmatrix}
1 & 0 & \alpha \\
0 & 1 & \beta \\
0 & 0 & 1
\end{pmatrix},
$$
where the column $(\alpha,\beta)^T$ varies in some finite subset $\Lambda$ of $k^2$.
Denote by $\Lambda_1$ the projection of $\Lambda$ to the first coordinate.

\begin{lemma}
The sets $\Lambda$ and $\Lambda_1$ have natural structures of vector spaces over the Galois field $\GF{2^n}$.
Moreover $\Lambda=(\Lambda_1,\Lambda_1)^T$ and $\dim_{\GF{2^n}}\Lambda=2\dim_{\GF{2^n}}\Lambda_1$.
\end{lemma}
\begin{proof}
Obviously, $N$ is an abelian group, and thus $\Lambda$ is a subgroup of $k^2$. It remains to show
that $\Lambda$ is preserved by multiplication by an element $e\in\GF{2^n}$. Note that, as always
in extensions with abelian $N$, the quotient group $\SL{2}{\GF{2^n}}$ acts on $N$ via conjugation.
In our case, this action is nothing else but the left multiplication of a column $(\alpha,\beta)^T$
by a matrix
$$\begin{pmatrix}
a & b \\
c & d
\end{pmatrix} \in \SL{2}{\GF{2^n}}.
$$
So, we have
$$\begin{pmatrix}
1 & e \\
0 & 1
\end{pmatrix},
\begin{pmatrix}
1 & 0 \\
e & 1
\end{pmatrix} \in \SL{2}{\GF{2^n}}, \;
\begin{pmatrix}
\alpha \\
\beta
\end{pmatrix} \in \Lambda \Rightarrow
$$
$$
\begin{pmatrix}
\alpha+e\beta \\
\beta
\end{pmatrix},
\begin{pmatrix}
\alpha \\
e\alpha+\beta
\end{pmatrix} \in \Lambda \Rightarrow
e \begin{pmatrix}
\beta \\
\alpha
\end{pmatrix} \in \Lambda.
$$
But
$$
\begin{pmatrix}
0 & 1 \\
1 & 0
\end{pmatrix} \in \SL{2}{\GF{2^n}} \Rightarrow
e \begin{pmatrix}
\alpha \\
\beta
\end{pmatrix} \in \Lambda.
$$

Multiplying a column $(\alpha,\beta)^T\in\Lambda$ by matrices from the subgroup
$\SL{2}{\GF{2}}<\SL{2}{\GF{2^n}}$, one readily checks that the set $\Lambda$ also contains
$(\alpha,0)^T$, $(0,\beta)^T$, $(0,\alpha)^T$, and $(\beta,0)^T$.
The remaining statements follow directly from this fact.
\end{proof}

The following proposition describes a convenient set of generators of the group $\SL{2}{\GF{2^n}}$.
\begin{proposition}
The group $\SL{2}{\GF{2^n}}$ is generated by the matrices
$$R=\begin{pmatrix}
e^{-1} & 0 \\
0 & e
\end{pmatrix}, \:
S=\begin{pmatrix}
1 & 1 \\
0 & 1
\end{pmatrix}, \:
T=\begin{pmatrix}
1 & 0 \\
1 & 1
\end{pmatrix},
$$
where $e$ is a generator of the multiplicative group $\GF{2^n}^*$ of the field $\GF{2^n}$.
\end{proposition}
\begin{proof}
It is well known (see, e.g., \cite[Chapter 1]{Bonnafe}) that $\SL{2}{\GF{2^n}}$ is generated by
its subgroup of diagonal matrices, the subgroup of upper triangular unipotent matrices, and the
element
$$STS=\begin{pmatrix}
0 & 1 \\
1 & 0
\end{pmatrix}.
$$
If we are given the elements $R$, $S$, $T$, we can get any matrix
$$\begin{pmatrix}
1 & e^r \\
0 & 1
\end{pmatrix}
$$
as $R^{-r/2}SR^{r/2}$, where
$$R^{r/2}=\begin{pmatrix}
e^{-r/2} & 0 \\
0 & e^{r/2}
\end{pmatrix}
$$
(recall that each element of $\GF{2^n}$ has a unique square root in $\GF{2^n}$).
\end{proof}

\begin{remark}
Note that the matrices $S$ and $T$ generate the group $\SL{2}{\GF{2}}$.
\end{remark}

In our next step we show that sequence \eqref{E:sequence} splits.
\begin{lemma}\label{L:generators}
After a change of the basis vector $e_3$, if necessary, we can assume that the group $G$
contains matrices
$$\tilde{S}=\begin{pmatrix}
1 & 1 & 0 \\
0 & 1 & 0 \\
0 & 0 & 1
\end{pmatrix}, \:
\tilde{T}=\begin{pmatrix}
1 & 0 & 0 \\
1 & 1 & 0 \\
0 & 0 & 1
\end{pmatrix},
$$
and one of the matrices
$$\tilde{R}=\begin{pmatrix}
e^{-1} & 0 & 1 \\
0 & e & e \\
0 & 0 & 1
\end{pmatrix} \text{ or }
\tilde{R}'=\begin{pmatrix}
e^{-1} & 0 & 0 \\
0 & e & 0 \\
0 & 0 & 1
\end{pmatrix}.
$$
\end{lemma}
\begin{proof}
As was shown in \cite[Lemma~4.4]{Stepanov}, the group $G$ contains transvections $\tilde{S}$
and $\tilde{T}$ that restrict to the elements $S$ and $T$ of $\SL{2}{\GF{2^n}}$ respectively.
Each of the transvections $\tilde{S}$ and $\tilde{T}$ fixes a plane, and these planes intersect
along a line not contained in the invariant subspace $W$. If we take $e_3$ to be any non-zero
vector from this line, then, in the basis $e_1$, $e_2$, $e_3$, $\tilde{S}$ and $\tilde{T}$ have
the desired matrices.

Now consider any element
$$\begin{pmatrix}
e^{-1} & 0 & \alpha \\
0 & e & \beta \\
0 & 0 & 1
\end{pmatrix} \in G
$$
that restricts to $R\in\SL{2}{\GF{2^n}}$. Using the matrix
$$\begin{pmatrix}
0 & 1 & 0 \\
1 & 0 & 0 \\
0 & 0 & 1
\end{pmatrix}=\tilde{S}\tilde{T}\tilde{S},
$$
we get one more matrix
$$\begin{pmatrix}
0 & 1 & 0 \\
1 & 0 & 0 \\
0 & 0 & 1
\end{pmatrix}
\begin{pmatrix}
e^{-1} & 0 & \alpha \\
0 & e & \beta \\
0 & 0 & 1
\end{pmatrix}
\begin{pmatrix}
0 & 1 & 0 \\
1 & 0 & 0 \\
0 & 0 & 1
\end{pmatrix}=
\begin{pmatrix}
e & 0 & \beta \\
0 & e^{-1} & \alpha \\
0 & 0 & 1
\end{pmatrix} \in G,
$$
thus
$$\begin{pmatrix}
e^{-1} & 0 & \alpha \\
0 & e & \beta \\
0 & 0 & 1
\end{pmatrix}
\begin{pmatrix}
e & 0 & \beta \\
0 & e^{-1} & \alpha \\
0 & 0 & 1
\end{pmatrix}=
\begin{pmatrix}
1 & 0 & e^{-1}\beta+\alpha \\
0 & 1 & e\alpha+\beta \\
0 & 0 & 1
\end{pmatrix} \in N.
$$
Further,
$$\begin{pmatrix}
1 & 0 & e^{-1}\beta+\alpha \\
0 & 1 & e\alpha+\beta \\
0 & 0 & 1
\end{pmatrix}
\begin{pmatrix}
e^{-1} & 0 & \alpha \\
0 & e & \beta \\
0 & 0 & 1
\end{pmatrix}^2=
\begin{pmatrix}
e^{-2} & 0 & e^{-1}(\alpha+\beta) \\
0 & e^2 & e(\alpha+\beta) \\
0 & 0 & 1
\end{pmatrix} \in G.
$$
The $2^{n-1}$-th power of the last matrix equals
$$\begin{pmatrix}
e^{-1} & 0 & (e^{1-2^n}+e^{3-2^n}+\cdots+e^{-1})(\alpha+\beta) \\
0 & e & (e^{2^n-1}+e^{2^n-3}+\cdots+e)(\alpha+\beta) \\
0 & 0 & 1
\end{pmatrix}=
$$
$$\begin{pmatrix}
e^{-1} & 0 & (e+1)^{-1}(\alpha+\beta) \\
0 & e & e(e+1)^{-1}(\alpha+\beta) \\
0 & 0 & 1
\end{pmatrix}.
$$
If $\alpha+\beta=0$, then we have found the matrix $\tilde{R}'\in G$. If $\alpha+\beta\ne 0$,
then, rescaling the basis vector $e_3$, we come to the matrix $\tilde{R}\in G$.
\end{proof}

\begin{lemma}\label{L:f}
Let $f\colon\GF{2^n}^2\to \GF{2^n}$ be a function defined by the formula
$$f(x,y)=1+x+y+x^{2^{n-1}}y^{2^{n-1}}.$$
Then, for all $a,b,c,d,p,q\in \GF{2^n}$ such that $ad+bc=1$, the following identity holds:
$$pf(a,b)+qf(c,d)+f(p,q)=f(pa+qc,pb+qd).$$
\end{lemma}
\begin{proof}
The lemma is proven by a straightforward substitution, bearing in mind that for any $x\in\GF{2^n}$ one
has $x^{2^n}=x$.
\end{proof}

\begin{corollary}\label{C:subgroup}
For all $\gamma\in k$ the set of matrices
$$H_\gamma=\left\{ \begin{pmatrix}
a & b & \gamma f(a,b) \\
c & d & \gamma f(c,d) \\
0 & 0 & 1
\end{pmatrix} \, \Bigg| \,
\begin{pmatrix}
a & b \\
c & d
\end{pmatrix} \in \SL{2}{\GF{2^n}} \right \}
$$
is a subgroup of $\GL(V)$ isomorphic to $\SL{2}{\GF{2^n}}$.
\end{corollary}

\begin{remark}
For any $\gamma\in k$, the map
$$\begin{pmatrix}
a & b \\
c & d
\end{pmatrix} \to
\begin{pmatrix}
\gamma f(a,b) \\
\gamma f(c,d)
\end{pmatrix}
$$
is a skew homomorphism from the group $\SL{2}{\GF{2^n}}$ to the additive group $k^2$,
generating the cohomology group $H^1(\SL{2}{\GF{2^n}},k^2)$, where $\SL{2}{\GF{2^n}}$
acts on the space $k^2$ of column vectors by left multiplication, see \cite{CPS}.
\end{remark}

\begin{proposition}\label{P:splitting}
The group $G$ contains one of the groups $H_0$ or $H_1$ defined in Corollary~\ref{C:subgroup}.
It follows, in particular, that $G$ is a semidirect product of the subgroups $N$ and $H_0$ ($H_1$),
that is, sequence \eqref{E:sequence} splits.
\end{proposition}
\begin{proof}
Indeed, it can be directly checked that $\tilde{R}$, $\tilde{S}$, $\tilde{T}\in H_1$, whereas
$\tilde{R}'$, $\tilde{S}$, $\tilde{T}\in H_0$.
\end{proof}

\begin{remark}
It is known that the second cohomology group $H^2(\SL{2}{\GF{2^n}})$ with coefficients in
the natural module is non-zero for $n>2$ (\cite[Proposition~4.4]{Chih-Han}), i.e., there exist
non-split extensions of $\SL{2}{\GF{2^n}}$ by $\GF{2^n}^2$. Our results mean that those
non-split extensions do not have representations of the type that we study in this section.
\end{remark}

\begin{remark}\label{R:fieldofdef}
Note that the groups $H_0$ and $H_1$ are defined over the field $\GF{2^n}$, i.e., the entries of
all the matrices of $H_0$ and $H_1$ belong to $\GF{2^n}$.
\end{remark}

\section{Proof of Theorem~\ref{T:even}: invariants}\label{S:codim1}

In this section we compute the invariants of the action of the group $G$ on the space $V\simeq k^3$.
We do this in two steps: first, we compute the invariants of the kernel $N$ and show that
$V/N$ is again isomorphic to $k^3$; then, we compute the action of the quotient group
$\SL{2}{\GF{2^n}}$ ($\simeq H_0$ or $H_1$, see Proposition~\ref{P:splitting}) on the invariants
of $N$ and show that also
$$V/G \simeq \frac{V/N}{H_0(H_1)} \simeq k^3.$$

We shall use the following criterion of Kemper.
\begin{proposition}[{\cite[Proposition~16]{Kemper}}]\label{P:degreecriterion}
Let $V$ be a vector space of dimension $n$ and $G<\GL(V)$ a finite group. Then $S(V^*)^G$
is polynomial if and only if there exist homogeneous invariants $f_1,\ldots,f_n\in S(V^*)^G$ of
degrees $d_1,\ldots,d_n$ such that $\prod_{i=1}^{n} d_i = |G|$ and the Jacobian determinant
$J=\det((\partial f_i/\partial x_j)_{i,j})$ is non-zero. If such $f_1,\dots,f_n$ exist, then they
generate freely the ring $S(V^*)^G$.
\end{proposition}

Recall that $N$ acts on $V$ by matrices
$$\begin{pmatrix}
1 & 0 & \alpha \\
0 & 1 & \beta \\
0 & 0 & 1
\end{pmatrix},
$$
where the column $(\alpha,\beta)^T$ runs over a finite dimensional $\GF{2^n}$-vector
space $\Lambda\subset k^2$.
Let $x$ ,$y$, $z$ be a basis of $V^*$ dual to the basis
$e_1$, $e_2$, $e_3$ of $V$ chosen in Section~\ref{S:extofSL}. Obviously, the polynomials
\begin{align*}
& f_x=\prod_{\alpha\in\Lambda_1} (x+\alpha z), \\
& f_y=\prod_{\alpha\in\Lambda_1} (y+\alpha z), \\
& f_z=z
\end{align*}
are invariant under the action of $N$.

\begin{lemma}\label{L:fxfy}
The polynomial $f_x$ ($f_y$) can involve $x$ ($y$) only in degrees $2^{mn}$, where
$0\leq m\leq d=\dim_{\GF{2^n}}\Lambda_1$.
\end{lemma}
\begin{proof}
Let $q=2^n$ and
$$
f'_x=\prod_{\alpha\in\Lambda_1}(x+\alpha).
$$
By the definition of the \emph{Dickson invariants} $c_m\in k$ (see, e.g., \cite[Section~8.1]{Benson}), we have
$$
f'_x=x^{q^d}+\sum_{m=0}^{d-1}c_mx^{q^m}.
$$
To conclude the proof, it remains to note that $f_x$ is obtained from $f'_x$ by ``homogenization'' with the help of $z$:
a monomial $x^k$ with $k\le q^d$ is replaces by $x^kz^{q^d-k}$.
\end{proof}

\begin{proposition}
The ring of invariants $S(V^*)^N$ is a polynomial ring generated by $f_x$, $f_y$, $f_z$.
\end{proposition}
\begin{proof}
We have $|N|=|\Lambda|=2^{2dn}=\deg f_x \cdot\deg f_y \cdot\deg f_z$. Thus, by
Proposition~\ref{P:degreecriterion}, we need to check only that the Jacobian is non-zero.
But, using Lemma~\ref{L:fxfy}, we get
\begin{align*}
J(f_x,f_y,f_z)= & \begin{vmatrix}
\left( \prod\limits_{\substack{\alpha\in\Lambda_1 \\ \alpha\ne 0}} \alpha \right)\cdot z^{2^{dn}-1} & 0
& \frac{\partial f_x}{\partial z} \\
0 & \left( \prod\limits_{\substack{\alpha\in\Lambda_1 \\ \alpha\ne 0}} \alpha \right)\cdot z^{2^{dn}-1}
& \frac{\partial f_y}{\partial z} \\
0 & 0 & 1
\end{vmatrix}= \\
 & \left( \left(\prod_{\substack{\alpha\in\Lambda_1 \\ \alpha\ne 0}} \alpha \right)\cdot z^{2^{dn}-1} \right)^2
 \ne 0.
\end{align*}
\end{proof}

Next we have to determine the action of the groups $H_0$ and $H_1$ on $f_x$, $f_y$, and $f_z$.
Let us begin with $H_0$. The generators of this group leave invariant the variable $z$ and are
defined over the field $\GF{2^n}$ (see Remark~\ref{R:fieldofdef}). From this and from
Lemma~\ref{L:fxfy} it follows that the action of $H_0$ on $f_x$, $f_y$, $f_z$ is linear,
that is, if $h\in H_0$, then it acts on the tuple $(f_x,f_y,f_z)$ by right matrix multiplication:
$$(f_x,f_y,f_z)\mapsto (f_x,f_y,f_z)\cdot h.$$
Therefore, in this case we can simply ignore the kernel $N$. Furthermore, since (the representation
of) the group $H_0$ is decomposable, the polynomiality of its ring of invariants has been
already established by Kemper and Malle \cite[Section~8]{KM}.

Now, consider the indecomposable group $H_1$. For the sake of clearness and simplicity, let
us start with the case when there is no kernel, i.e., $N=\{0\}$ and $H_1=G$. We shall need the
invariants of the action of $\SL{2}{\GF{2^n}}$ on its natural module. Let $W=k^2$ be a
$2$-dimensional space of column vectors over a field $k$ containing $\GF{2^n}$, and let
the group $\SL{2}{\GF{2^n}}$ act on $W$ by left matrix multiplication. Denote by $W^*$ the
dual space. The Dickson invariants (see, e.g., \cite[Proposition~8.1.3]{Benson}) are
$$c_0=\prod_{\substack{l\in W^* \\ l\ne 0}} l,$$
and
$$c_1=\sum_{\substack{U\subseteq W \\ \dim U=1}}\, \prod_{\substack{l\in W^* \\ l|_U\ne 0}} l$$
(for $c_1$ the sum is taken over all $1$-dimensional subspaces of $W$, and the product over all linear
forms that restrict to a non-zero form on $U$). It is not hard to see that there exists a root of degree
$2^n-1$ of the polynomial $c_0$, that is, $\exists u\in S(W^*)\colon u^{2^n-1}=c_0$, and that
$u$ and $c_1$ are $\SL{2}{\GF{2^n}}$-invariant.

\begin{theorem}[{\cite[Theorem~8.2.1]{Benson}}]
The ring of invariants of $\SL{2}{\GF{2^n}}$ on $W$ is polynomial and generated by $u$ and $c_1$.
\end{theorem}

Let us come back to our group $G=H_1$ and space $V$. Since we have a $G$-invariant subspace
$W$, the restriction to $W$ of each invariant of $G$ is an $\SL{2}{\GF{2^n}}$-invariant. Thus we
have a homomorphism $S(V^*)^G\to S(W^*)^{\SL{2}{\GF{2^n}}}$ of invariant rings. In a
general modular case, there is no reason for such a homomorphism to be surjective. However,
we shall see that we do have a surjection in our case and this will be a crucial step in computing
the invariants of $G$.

\begin{lemma}\label{L:lift}
Let $G$, $V$, and $W$ be as defined above. Then the restriction homomorphism $S(V^*)^G\to
S(W^*)^{\SL{2}{\GF{2^n}}}$ is surjective.
\end{lemma}
\begin{proof}
It is sufficient to lift to the ring $S(V^*)^G$ the invariants $u,c_1\in S(W^*)^{\SL{2}{\GF{2^n}}}$.
We shall work in the explicit coordinates $x,y,z$ defined after Proposition~\ref{P:degreecriterion},
so that any linear form $l\in V^*$ can be written as $l=ax+by+cz$, $a,b,c\in k$. Together with
the function $f$ (see Lemma~\ref{L:f}), consider also a function $g\colon\GF{2^n}^{2}\to
\GF{2^n}$:
$$g(x,y)=f(x,y)+1=x+y+x^{2^{n-1}}y^{2^{n-1}}.$$
It follows from Lemma~\ref{L:f} that g has the following property: for all $a,b,c,d,p,q\in \GF{2^n}$,
if $ad+bc=1$, then
\begin{equation}\label{E:g}
pf(a,b)+qf(c,d)+g(p,q)=g(pa+qc,pb+qd).
\end{equation}
Note also that $g$ is a homogeneous function of degree $1$ on $\GF{2^n}^{2}$, i.e.,
\begin{equation}\label{E:ghomogen}
\forall a,b,t\in \GF{2^n} \quad g(ta,tb)=tg(a,b).
\end{equation}

Now, let us lift each linear form $l=ax+by\in W^*$ to $V^*$ by the formula $\tilde{l}=
ax+by+g(a,b)z$ and define
\begin{align*}
\tilde{c}_0 &=\prod_{\substack{l\in W^* \\ l\ne 0}} \tilde{l}, \\
\tilde{c}_1 &=\sum_{\substack{U\subseteq W \\ \dim U=1}} \,
\prod_{\substack{l\in W^* \\ l|_{U}\ne 0}} \tilde{l}.
\end{align*}
Property \eqref{E:g} implies that both $\tilde{c}_0$ and $\tilde{c}_1$ are $G$-invariant. Obviously,
$\tilde{c}_0|_W=c_0$, $\tilde{c}_1|_W=c_1$. But, using property \eqref{E:ghomogen}, one readily
shows that $\tilde{c}_0$ admits a root of degree $2^n-1$, i.e., there exists $\tilde{u}\in S(V^*)$
such that $\tilde{u}^{2^n-1}=\tilde{c}_0$. Moreover, this $\tilde{u}$ is $G$-invariant and restricts
to $u\in S(W^*)^{\SL{2}{\GF{2^n}}}$.
\end{proof}

\begin{proposition}\label{P:invofH}
The ring of invariants $S(V^*)^G$ (for $G=H_1$) is polynomial and generated by (algebraically
independent) invariants $\tilde{u}$, $\tilde{c}_1$, $z$, where $\tilde{u}$ and $\tilde{c}_1$ are
defined in the proof of Lemma~\ref{L:lift}.
\end{proposition}
\begin{proof}
Let $\tilde{c}\in S(V^*)^G$ be an arbitrary homogeneous invariant. Let $c=\tilde{c}|_W$.
Write $c$ as a polynomial of $u$ and $c_1$:
$$c=h(u,c_1).$$
The $G$-invariant $\tilde{c}-h(\tilde{u},\tilde{c}_1)$ vanishes on $W$, thus it is divisible by $z$.
But since $z$ is also a $G$-invariant, so is the polynomial
$$\tilde{c}'=(\tilde{c}-h(\tilde{u},\tilde{c}_1))/z.$$
The degree of $\tilde{c}'$ is strictly less than that of $\tilde{c}$, so, proceeding by induction, we
express $\tilde{c}$ through $\tilde{u}$, $\tilde{c}_1$, and $z$.
\end{proof}

Now we return to the general case of a non-zero kernel $N$. A direct calculation with a use of
Lemma~\ref{L:fxfy} shows that the two generators $\tilde{S}$, $\tilde{T}$ (see
Lemma~\ref{L:generators}) of the group $H_1$ act on the basis invariants $f_x$, $f_y$, $f_z$
of $N$ by the formulae
\begin{align*}
f_x\cdot\tilde{S}=f_x+f_y,  \quad f_y\cdot\tilde{S}=f_y, \quad f_z\cdot\tilde{S}=f_z, \\
f_x\cdot\tilde{T}=f_x, \quad f_y\cdot\tilde{T}=f_x+f_y, \quad f_z\cdot\tilde{T}=f_z,
\end{align*}
i.e., their action is linear. It follows from Lemma~\ref{L:fxfy} that the third generator $\tilde{R}$ acts by the formulae
$$f_x\cdot\tilde{R}=e^{-1}f_x+\alpha z^{2^{dn}},\quad f_y\cdot\tilde{R}=ef_x+e\alpha z^{2^{dn}},
\quad f_z\cdot\tilde{R}=f_z,$$
where $\alpha\in k$. 
It can happen that $\alpha=0$, so that the action of $H_1$ on $V/N$ is linear (in coordinates
$f_x$, $f_y$, $f_z$) and decomposable. But then again by the results of Kemper and Malle the
ring of invariants $S((V/N)^*)^{H_1}=S(V^*)^G$ is polynomial. In general, the coefficient $\alpha$
does not vanish and the action of $\tilde{R}$ becomes non-linear. Still, it is possible to follow
the argument of Lemma~\ref{L:lift} and Proposition~\ref{P:invofH}.

Note that the equation $f_z=z=0$ defines an invariant subspace $W/N$ of the quotient $V/N$ (which
we consider as a vector space isomorphic to $k^3$, the isomorphism being defined by the
functions $f_x$, $f_y$, $f_z$). The action of $H_1$ on $W/N$ is the natural action of $\SL{2}{\GF{2^n}}$.
So, let $u$ and $c_1$ be the basis invariants of $\SL{2}{\GF{2^n}}$, but now considered as
functions of $f_x,f_y,f_z$. Repeating the proof of Lemma~\ref{L:lift} with $f_x$ in place of $x$,
$f_y$ in place of $y$, and $\alpha z^{2^{dn}}$ in place of $z$, we find some liftings $\bar{u}$
and $\bar{c}_1$ of $u$ and $c_1$ to the ring of invariants $S(V^*)^G$.

The following proposition finishes the proof of Theorem~\ref{T:even}.
\begin{proposition}
The ring of invariants $S(V^*)^G=S((V/N)^*)^{H_1}$ is polynomial and generated by (algebraically
independent) invariants $\bar{u}$, $\bar{c}_1$, $z$.
\end{proposition}
\begin{proof}
This proposition is proven by argument similar to the proof of Proposition~\ref{P:invofH}.
\end{proof}

\end{document}